\documentclass[11pt]{article}

\usepackage{comment}
\usepackage{amsmath, amssymb, latexsym, amsthm}
\usepackage{hyperref}

\newtheorem{theorem}{Theorem} 
\newtheorem{cor}[theorem]{Corollary}
\newtheorem{claim}{Claim}
\newtheorem*{rem}{Remark}

\newcommand{\eps}{\varepsilon}
\newcommand{\cF}{\mathcal{F}}
\newcommand{\cB}{\mathcal{B}}
\newcommand{\Nat}{\mathbb{N}}

\DeclareMathOperator\disp{disp}
\def\const{1920}
\begin{document}

\title{A tight lower bound on the minimal dispersion}

\author{M. Tr\"odler\thanks{Department of Mathematics, Faculty of Nuclear Sciences and Physical Engineering,
Czech Technical University, Trojanova 13, 12000 Praha, Czech Republic.}
\thanks{E-mail: \href{mailto:trodlmat@fjfi.cvut.cz}{trodlmat@fjfi.cvut.cz}.
The work of this author has been supported by the grant P202/23/04720S of the Grant Agency of the Czech Republic},
\and J. Volec\footnotemark[1]
\thanks{E-mail: \href{mailto:jan.volec@fjfi.cvut.cz}{jan.volec@fjfi.cvut.cz}.
The work of this author has been supported by the grant 23-06815M of the Grant Agency of the Czech Republic},
\and and J. Vyb\'\i ral\footnotemark[1]
\thanks{Corresponding author. E-mail: \href{mailto:jan.vybiral@fjfi.cvut.cz}{jan.vybiral@fjfi.cvut.cz}.
The work of this author has been supported by the grant P202/23/04720S of the Grant Agency of the Czech Republic}
}
\date{}
\maketitle
\begin{abstract}
We give a new lower bound for the minimal dispersion of a point set in the unit cube and its inverse function in the high dimension regime.
This is done by considering only a very small class of test boxes, which allows us to reduce bounding the dispersion to a problem in extremal set theory.
Specifically, we translate a lower bound on the size of $r$-cover-free families
to a lower bound on the inverse of the minimal dispersion of a point set.
The lower bound we obtain matches the recently obtained upper bound on the minimal dispersion up to logarithmic terms.
\end{abstract}

\section{Introduction and the Main Result}

For a given subset of the $d$-dimensional unit cube $X\subset [0,1]^d$, there are various different ways how to measure whether $X$ is well spread over the unit cube.
Based on the previous work of Hlawka \cite{Hlawka} and Niederreiter \cite{Nieder}, Rote and Tichy~\cite{RT96} introduced in 1996 a concept called the dispersion of $X$.
The \emph{dispersion of $X$} is the volume of the largest axis-parallel box in $[0,1]^d$, which does not intersect $X$, i.e.,
\[
\disp(X)=\sup_{B:B\cap X=\emptyset} |B|,
\]
where the supremum runs over all the boxes $B=\prod_{i=1}^d (a_i,b_i)$, where $0\le a_i<b_i\le 1$ for all $i\in\{1,\dots,d\},$ and $|B|$ stands for the volume of $B$.
Intuitively, the smaller the dispersion of a point set is, the better spread over the unit cube the points have.
We are therefore interested in point sets that have simultaneously small cardinality and small dispersion.

Given $n \in \Nat$, we denote the minimal dispersion of a point set that has cardinality $n$ by
\[
\disp^*(n,d)=\inf_{\substack{X\subset [0,1]^d\\\#X=n}}\disp(X).
\]
For a fixed $d \in \Nat$ and $\eps\in(0,1)$, we define the inverse function of the minimal dispersion in $[0,1]^d$ by
\begin{align*}
N(\varepsilon,d)&=\min\{n\in\Nat: \disp^*(n,d)\le \varepsilon\}\\
&=\min\{n\in\Nat: \exists X\subset [0,1]^d\ \text{with}\ \#X=n\ \text{and}\ \disp(X)\le\varepsilon\}
.
\end{align*}

The study of the minimal dispersion $\disp^*(n,d)$ as well as its inverse function $N(\varepsilon,d)$
recently attracted a remarkable attention in the mathematical community.
The pigeonhole principle implies the trivial lower bound  $\disp^*(n,d)\ge \frac{1}{n+1}$,
which translates into \[N(\varepsilon,d)\ge \frac1\varepsilon-1.\]
This bound was complemented in \cite{RT96} and then further improved by Bukh and Chao~\cite{BC} to
\begin{equation}\label{eq:BC}
N(\varepsilon,d)\le C \cdot \frac{d^2\log d}{\varepsilon}.
\end{equation}
Let us note that further upper and lower bounds for different regimes of $\varepsilon$ and $d$ were obtained also in \cite{Dum, Lit, LL, Mac, Rudolf},
and the dispersion of certain specific sets was studied in \cite{BH, HKKR, Krieg, Kritz, LW, Teml,Mario}.

The first non-trivial lower bound on $N(\varepsilon,d)$ that exhibits a growth if $d$ increases was achieved by Aistleitner, Hinrichs and Rudolf~\cite{AHR}.
Specifically, they showed that
\begin{equation}\label{eq:AHR}
N(\varepsilon,d)\ge \frac{\log_2 d}{8\varepsilon}
\end{equation}
for every $d\ge 2$ and every $0<\varepsilon<1/4.$
The results mentioned so far suggest that $N(\varepsilon,d)$ depends linearly on $1/\varepsilon$ and it only remains to determine the $d$-dependence.

In another line of research, Sosnovec~\cite{Sosnovec} showed that $N(\varepsilon,d)\le C_\varepsilon \log d.$
The dependence of $C_\varepsilon$ was not optimized in \cite{Sosnovec}, and a brief inspection of the proof reveals that it is actually super-exponential in $1/\varepsilon$.
Nevertheless, it matches \eqref{eq:AHR} when it comes to the dependence on $d$.
Let us mention that the point set with small dispersion was generated in \cite{Sosnovec}
by sampling the coordinates of the points independently at random from a uniform grid in $[0,1].$
An improved analysis of the same method was given in \cite{UV}, which together with a further improvement due to Litvak~\cite{Lit} yields
\begin{equation}\label{eq:UVL}
N(\varepsilon,d)\le C\cdot \frac{\log d \cdot \log(\frac 1\varepsilon)}{\varepsilon^2}.
\end{equation}
The comparison between \eqref{eq:BC} and \eqref{eq:UVL} is not easy. Depending on the relation between $d$ and $1/\varepsilon$, one or the other might perform better.
Although the analysis in~\cite{Lit,UV} regarding the dependence on $1/\varepsilon$ was much finer than the one in~\cite{Sosnovec},
it was quite natural to assume that the second power of $1/\varepsilon$ in \eqref{eq:UVL} is an artefact of the proof method and that the bound could be potentially further improved even when $\varepsilon$ is moderately large.

The main result of this work is to show that, surprisingly, this is not the case and the second power of $1/\varepsilon$ in \eqref{eq:UVL} is optimal when $\varepsilon$ is sufficiently large.
\begin{theorem}\label{thm:main}
There is an absolute constant $c>0$, such that the following statement is true. For any integer $d\ge 2$ and any real $\eps$ satisfying $\frac14 \ge \eps \ge \frac1{4\sqrt{d}}$,
it holds that
\begin{equation}\label{eq:mainthm}
N(\eps,d) >  \frac{ c\,\log d }{ \eps^2 \cdot \log{\frac1\eps} } 
\,.
\end{equation}
\end{theorem}

The proof of this result is motivated by the insight gained by a detailed inspection of the proof given in \cite{UV}.
For a fixed box $B$, the sampling from a finite grid, which is bounded away of zero and one, naturally splits the individual coordinates into two groups ---
those, where the corresponding interval $(a_i,b_i)$ is large enough and covers the whole grid, and the coordinates, where this is not the case.
It turned out in \cite{UV} that the boxes with a large number of the coordinates inside the second group are the most difficult to hit by a randomly generated point.
Therefore, we choose a very small class of test boxes when compared to the number of all the axis-parallel boxes, and consider any point set that hits all these test boxes.
This naturally reduces the question under study to a well-known combinatorial problem called \emph{$r$-cover-free systems}, see Section~\ref{sec:proof} for the definition.
The key tool in our proof is a lower bound of Alon and Asodi~\cite{AA} on the minimum size of {$r$-cover-free system}.

Reformulated in the language of minimal dispersion, Theorem \ref{thm:main} states the following.
\begin{cor}There are two absolute constants $c_1,c_2>0$ such that for every positive integers $d$ and $n$ with $d\ge 2$ and $2\log d\le n\le c_1 d$ it holds that
\[
\disp^*(n,d)\ge c_2 \biggl(\frac{\log d}{n}\biggr)^{1/2}\cdot \biggl(\log\Bigl(\frac{n}{\log d}\Bigr)\biggr)^{-1/2}.
\]
\end{cor}

Before we present our proof of Theorem \ref{thm:main}, let us add a few remarks.
\begin{rem}
\begin{enumerate}
\item It follows from the proof of Theorem~\ref{thm:main} that one can actually choose $c=1/\const$, but we did not attempt to optimize this constant.
\item The essentially tight factor $1/\eps^2$ in the lower bound for $N(\varepsilon,d)$ in \eqref{eq:mainthm} is rather surprising.
It seems that a logarithmic dependence on $d$ and a linear dependence on $1/\varepsilon$ exclude each other.
What might be even more surprising, this lower bound is obtained by considering only a very small class of axis-parallel test boxes.
\item We leave it as an open problem whether the proof method can be further improved, possibly extending \eqref{eq:mainthm} to smaller values of $\varepsilon$.
\end{enumerate}
\end{rem}

\section{Proof of Theorem \ref{thm:main}}\label{sec:proof}

We start by fixing some notation that will be used in the proof.
For a finite set $X$, we denote its size by $\# X$.
For a positive integer $d$, we denote by $[d]$ the set $\{1,2,\ldots,d\}$.
Additionally, for a non-negative integer $k$, we write $\binom{[d]}k$ to denote the collection of all the $k$-element subsets of $[d]$.
Given a point $x \in [0,1]^d$ and an integer $i \in [d]$, we denote by $(x)_i$ the $i$-th coordinate of $x$.

By focusing on a very specific set of axis-parallel boxes of volume at least $\eps$, we are able to translate bounding $N(\eps,d)$ from below to a question in extremal set theory.
We say that a family $\cF$ of subsets of a ground set $X$ is \emph{$r$-cover-free} if no $F \in \cF$ is contained in the union of any $r$ sets from $\cF \setminus \{F\}$.
Based on a previous work of Ruszink\'o \cite{Rusz}, Alon and Asodi~\cite{AA} proved the following lower-bound on the size of the ground set of an $r$-cover-free family.
\begin{theorem}[{\cite[Lemma 2.8]{AA}}]
\label{thm:AlonAsodi}
Let $\cF$ be a family of $d$ subsets of a ground set $X$.
If $\cF$ is $r$-cover-free, where $2 \le r \le 2\sqrt{d}$, then \[
\# X > \frac1{10}\cdot \frac{r^2 \log\left(d-\frac r2\right)}{\log r}
\,.\]
\end{theorem}

We now proceed with the proof of Theorem~\ref{thm:main}.
Fix $d\ge2$ and $\eps$ such that $\frac14 \ge \eps \ge \frac1{4\sqrt{d}}$,
and let $k$ be the unique integer satisfying $2^{-(k+1)} < \eps \le 2^{-k}$.
Note that $k\ge2$.

Our plan is to define a collection $\cB$ of $\binom{d}{2^{k-2}}\left(d-2^{k-2}\right)$ boxes such that
each box has one coordinate where all its points are close to zero,
while in some other $2^{k-2}$ coordinates all the points of the box are bounded away from zero.
Next, we will sort out the points of $X$ with at least one coordinate close to zero into (not necessarily disjoint) sets $F_1,F_2,\ldots,F_d$.
All of this is done in such a way that the sets $F_1,F_2,\ldots,F_d$ are $2^{k-2}$-cover-free, thus Theorem~\ref{thm:AlonAsodi} yields the desired lower bound on $X$.

We now give all the details.
For a given set $A \subseteq [d]$ and $j \in [d]\setminus A$, let $B^{A,j} \subseteq [0,1]^d$ be the axes-parallel box
defined as
\[
B^{A,j}:=I_1 \times I_2 \times \cdots \times I_d, \quad \mbox{where }  \begin{cases}
  I_j =(0,2^{1-k}), \\
  I_i =(2^{1-k},1) \quad \mbox {for } i \in A, \\
  I_i =(0,1) \quad \mbox{for } i \in [d] \setminus \left(A \cup\{j\}\right).
\end{cases}
\]
Let $\cB := \left\{ B^{A,j} : A \in \binom{[d]}{2^{k-2}} \mbox { and } j\in[d]\setminus A \right\}.$ 
The following claim yields that any set of points from $[0,1]^d$ with dispersion at most $\eps$ must have at least one point inside every $B \in \cB$.
\begin{claim}
$\left|B^{A,j}\right| \ge \eps$ for any $A \subseteq [d]$ of size at most $2^{k-2}$ and any $j\in [d]\setminus A$.
\end{claim}
\begin{proof}
Since $2^{-2x}$ is convex, it holds that $(1-x) \ge 2^{-2x}$ for every $x\in[0,1/2]$.
Therefore,
\[
\left|B^{A,j}\right| = \left(1-2^{1-k}\right)^{\# A} \cdot 2^{1-k} \ge \left(2^{-2^{2-k}}\right)^{\# A} \cdot 2^{1 - k} \ge 2^{-k} \ge \eps
\,,\]
which finishes the proof of the claim.
\end{proof}

Fix a finite set of points $X \subseteq [0,1]^d$ that intersects every $B \in \cB$.
For every $j \in [d]$, we define an auxiliary set $F_j$ of the points with a small $j$-th coordinate.
Specifically, let $F_j := \{x \in X: (x)_j < 2^{1-k}\}$.
In order to lower bound $\# X$ using Theorem~\ref{thm:AlonAsodi}, we establish the following claim.
\begin{claim}
The family $\left\{F_1,F_2,\ldots,F_d\right\}$ is $2^{k-2}$-cover-free.
\end{claim}
\begin{proof}
Suppose there is a set $A \in \binom{[d]}{2^{k-2}}$ and an integer $j \in [d]\setminus A$ such
that \[
F_j \subseteq \bigcup_{i \in A} F_i
\,.\]
However, there must exist a point $x \in X \cap B^{A,j}$.
Since $(x)_j < 2^{1-k}$ and $(x)_i > 2^{1-k}$ for all $i \in A$, we conclude that $x$ is an element of $F_j \setminus \bigcup_{i \in A} F_i$, which is a contradiction.
\end{proof}

If $k=2$, then $\eps \in (1/8,1/4]$ and $\# X \ge \log_2(d) > \frac3{64}\cdot\frac{\log d }{\eps^2 \cdot \log\left(1/\eps\right) }$ by the $1$-cover-freeness.
For $k\ge3$, we have that $2^{k-2} \le \frac1{4\eps} \le \sqrt{d}$, thus
Theorem~\ref{thm:AlonAsodi} applied to the family $\left\{F_1,\ldots,F_d\right\}$ and $r = 2^{k-2}$
readily yields that
\[
\# X > \frac1{10} \cdot \frac{2^{2k-4} \cdot \log\left(d-2^{k-3}\right)}{\log \left(2^{k-2}\right)} \ge \frac1{10} \cdot \frac{2^{2k-4} \cdot \log\left(d-\sqrt{d}/2\right)}{\log \left(2^{k-2}\right)}
\,.\]
Since $\frac1{8\eps} < 2^{k-2} < \frac1{\eps}$ and $\log(d-\sqrt{d}/2) \ge \frac{\log{d}}{3}$ for all $d\ge2$, we conclude that
\[
\# X > \frac1\const \cdot \frac{ \log d }{ \eps^2 \cdot \log{\frac1\eps} } 
\,.\]
This finishes the proof of Theorem~\ref{thm:main}.

\vskip.5cm

{\bf Acknowledgements.}
We are extremely grateful to Noga Alon for pointing out to us the relevance of the notion of $r$-free-cover systems as well as the papers \cite{AA} and \cite{Rusz}.
We thank Noah Kravitz, Alexander Litvak, Vojt\v{e}ch R\"odl, Marcelo Tadeu de S\'a Oliveira Sales, and Stefan Steinerberger for fruitful discussions,
and the anonymous referees for helping us to improve the presentation of this paper.
We also gratefully acknowledge the support of the Leibniz Center for Informatics,
where several discussions about this problem were held during the Dagstuhl Seminar ``Algorithms and Complexity for Continuous Problems'' (Seminar ID 23351).


\begin{thebibliography}{99}
\bibitem{AHR} C. Aistleitner, A. Hinrichs and D. Rudolf, \emph{On the size of the largest empty box amidst a point set}, Discrete Appl. Math. 230 (2017), 146--150

\bibitem{AA} N. Alon and V. Asodi, \emph{Learning a hidden subgraph}, SIAM J. Discrete Math. 18 (2005), 697--712

\bibitem{BH} S. Breneis and A. Hinrichs, \emph{Fibonacci lattices have minimal dispersion on the two-dimensional torus}, in: Dmitriy Bilyk, Josef
Dick, Friedrich Pillichshammer (Eds.), Discrepancy Theory, De Gruyter, 2020, pp. 117–132

\bibitem{BC} B. Bukh and T. Chao, \emph{Empty axis-parallel boxes}, Int. Math. Res. Notices 18 (2022), 13811--13828

\bibitem{Dum} A. Dumitrescu and M. Jiang, \emph{On the largest empty axis-parallel box amidst $n$ points}, Algorithmica 66 (2013), 225--248

\bibitem{HKKR} A. Hinrichs, D. Krieg, R. J. Kunsch, and D. Rudolf, \emph{Expected dispersion of uniformly distributed points},
J. Compl. 61 (2020), 101483

\bibitem{Hlawka} E. Hlawka, \emph{Absch\"atzung von trigonometrischen Summen mittels diophantischer Approximationen}, \"Osterreich. Akad. Wiss. Math.-
Naturwiss. Kl. S.-B. II, 185 (1976), 43--50

\bibitem{Krieg} D. Krieg, \emph{On the dispersion of sparse grids}, J. Compl. 45 (2018), 115--119

\bibitem{Kritz} R. Kritzinger, \emph{Dispersion of digital $(0, m, 2)$-nets}, Monatshefte Math.  195 (1) (2021), 155--171

\bibitem{LW} T. Lachmann and J. Wiart, \emph{The area of empty axis-parallel boxes amidst 2-dimensional lattice points}, J. Compl. 76 (2023), 101724

\bibitem{Lit} A. E. Litvak, \emph{A remark on the minimal dispersion}, Commun. Contemp. Math. 23 (2021), 2050060

\bibitem{LL} A. E. Litvak and G. V. Livshyts, \emph{New bounds on the minimal dispersion}, J. Compl. 72 (2022), 101648

\bibitem{Mac} K. MacKay, \emph{Minimal dispersion of large volume boxes in the cube}, J. Compl. 72 (2022), 101650

\bibitem{Nieder}  H. Niederreiter, \emph{A quasi-Monte Carlo method for the approximate computation of the extreme values of a function},
In: P. Erd\"os, L. Alp\'ar, G. Hal\'asz, A. S\'ark\"ozy, (eds.), Birkh\"auser Basel, Basel (1983), 523-529

\bibitem{RT96} G. Rote and R. F. Tichy, \emph{Quasi-Monte Carlo methods and the dispersion of point sequences}, Math. Comput. Modelling 23 (1996), 9--23

\bibitem{Rudolf} D. Rudolf, \emph{An upper bound of the minimal dispersion via delta covers}, Contemporary Computational Mathematics - A Celebration of the 80th Birthday
of Ian Sloan, Springer-Verlag (2018), 1099--1108

\bibitem{Rusz} M. Ruszink\'o, \emph{On the upper bound of the size of the $r$-cover-free families}, J. Combin. Theory Ser. A 66 (1994), 302--310

\bibitem{Sosnovec} J. Sosnovec, \emph{A note on the minimal dispersion of point sets in the unit cube}, Eur. J. Comb. 69 (2018), 255--259

\bibitem{Teml} V. N. Temlyakov, \emph{Dispersion of the Fibonacci and the Frolov point sets}, preprint, 2017, arXiv:1709.08158

\bibitem{Mario} M. Ullrich, \emph{A note on the dispersion of admissible lattices}, Discrete Appl. Math. 257 (2019), 385–387

\bibitem{UV} M. Ullrich and J. Vyb\'\i ral, \emph{An upper bound on the minimal dispersion}, J. Compl. 45 (2018), 120--126

\end{thebibliography}
\end{document}